\documentclass[a4paper,10pt]{amsart}
\usepackage[utf8]{inputenc}

\usepackage[alphabetic,initials]{amsrefs}
\usepackage{amsmath}
\usepackage{amsthm}
\usepackage{amssymb}
\usepackage{mathrsfs}
\usepackage[all]{xy}

\newcommand{\SDR}[5]{\xymatrix{*[r]{#1} \ar@<1ex>[r]^-{#3} \ar@(ul,dl)[]_{#5} & #2 \ar@<1ex>[l]^-{#4}}}
\newcommand{\bigSDR}[5]{\xymatrix{*[r]{#1} \ar@<1ex>[rr]^-{#3} \ar@(ul,dl)[]_{#5} && #2 \ar@<1ex>[ll]^-{#4}}}
\newcommand{\bigbigSDR}[5]{\xymatrix{*[r]{#1} \ar@<1ex>[rrr]^-{#3} \ar@(ul,dl)[]_{#5} &&& #2 \ar@<1ex>[lll]^-{#4}}}
\newcommand{\kk}{\Bbbk}
\newcommand{\tensor}{\otimes}
\newcommand{\CC}{\mathbb{C}}
\newcommand{\QQ}{\mathbb{Q}}

\newcommand{\ZZ}{\mathbb{Z}}

\newcommand{\Hom}{\operatorname{Hom}}

\newcommand{\as}{\text{\normalfont{<}}} 

\newcommand{\gl}{\mathfrak{g}}

\newcommand{\im}{\operatorname{im}}
\newcommand{\coker}{\operatorname{coker}}

\newcommand{\rank}{\operatorname{rank}}
\newcommand{\CP}[1]{\CC \hspace{-2.5pt} \operatorname{P}^{#1}}

\newcommand{\set}[2]{\left\{ #1 \mid #2 \right\}}

\newcommand{\LL}{\mathbb{L}}

\newtheorem{theorem}{Theorem}
\newtheorem{proposition}[theorem]{Proposition}
\newtheorem{lemma}[theorem]{Lemma}
\newtheorem{corollary}[theorem]{Corollary}

\theoremstyle{definition}
\newtheorem*{definition*}{Definition}
\newtheorem{definition}[theorem]{Definition}
\newtheorem{example}[theorem]{Example}
\newtheorem{remark}[theorem]{Remark}

\numberwithin{theorem}{section}

\title{Koszul $A_\infty$-algebras and free loop space homology}
\author{Alexander Berglund and Kaj B\"orjeson}

\begin{document}

\begin{abstract}
We introduce a notion of Koszul $A_\infty$-algebra that generalizes Priddy's notion of a Koszul algebra and we use it to construct small $A_\infty$-algebra models for Hochschild cochains. As an application, this yields new techniques for computing free loop space homology algebras of manifolds that are either formal or coformal (over a field or over the integers). We illustrate these techniques in two examples.
\end{abstract}

\maketitle

\section{Introduction}
Koszul algebras were introduced by Priddy \cite{Priddy} and have since then played an important role in algebraic topology, representation theory and homological algebra, see \cite{PolishchukPositselski,LodayVallette} for introductory accounts.~$A_\infty$-algebras are generalizations of associative algebras where one relaxes the associativity constraint up to (coherent) homotopy. They were introduced by Stasheff \cite{Stasheff} in the study of homotopy associative $H$-spaces, but have found applications in many other branches of mathematics, see \cite{Keller} for a readable introduction.

In this paper, we introduce a notion of Koszul $A_\infty$-algebra that generalizes Priddy's notion of a Koszul algebra (see Definition \ref{def:koszul}).
The principal technical result is the following characterization of Koszul $A_\infty$-algebras, which generalizes a known characterization of Koszul algebras (see \cite[Remark 2.13]{BerglundBorjeson}). Recall that a dg coalgebra is called \emph{formal} if it is quasi-isomorphic to its homology through maps of dg coalgebras.

\begin{theorem}
An $A_\infty$-algebra $A$ is quasi-isomorphic to a Koszul $A_\infty$-algebra if and only if the bar construction $BA$ is formal as a dg coalgebra. Moreover, every Koszul $A_\infty$-algebra is quasi-isomorphic to a minimal Koszul $A_\infty$-algebra.
\end{theorem}
Furthermore, we show that several of the main features of Koszul algebras carry over to Koszul $A_\infty$-algebras.
In particular, if $A$ is Koszul, then the homology of the bar construction $H_*(BA)$ may be computed as the Koszul dual coalgebra $A^\as$, which can be read off from a presentation for $A$. We also construct small $A_\infty$-algebra models for the Hochschild cochains on Koszul $A_\infty$-algebras that facilitate the computation of Hochschild cohomology (see Theorem \ref{thm:koszulhochschild}).

Our work is motivated by the problem of computing the homology of free loop spaces of manifolds, together with algebraic structure such as the Chas-Sullivan loop product \cite{ChasSullivan}. Free loop space homology is interesting because of its relation to closed geodesics (see e.g.~\cite[Chapter 5]{FelixOpreaTanre} and \cite{GoreskyHingston}) and because it provides potentially interesting examples of `homological conformal field theories' (see \cite{Godin}). In previous work \cite{BerglundBorjeson}, we explained how Koszul algebras can be used to compute $H_*(LM)$ for manifolds $M$ that are both formal and coformal. While there are interesting examples of formal and coformal manifolds, this is a rather restrictive constraint. The new notion of Koszul $A_\infty$-algebras introduced in this paper allows us to treat manifolds that are \emph{either} formal or coformal, not necessarily both. Our main results are summarized by the following theorem.
\begin{theorem}
Let $M$ be a simply connected topological space and let $\kk$ be a field.
\begin{enumerate}
\item The space $M$ is formal over $\kk$ if and only if the homology of the based loop space $H_*(\Omega M;\kk)$ admits a minimal Koszul $A_\infty$-algebra structure making it quasi-isomorphic to $C_*(\Omega M;\kk)$. In this situation, the homology of $M$ is isomorphic to the Koszul dual coalgebra,
$$H_*(M;\kk) \cong H_*(\Omega M;\kk)^\as.$$

\item The space $M$ is coformal over $\kk$ if and only if its homology $H_*(M;\kk)$ admits a minimal Koszul $A_\infty$-coalgebra structure making it quasi-isomorphic to $C_*(M;\kk)$. In this situation, the homology of the based loop space is isomorphic to the Koszul dual algebra,
$$H_*(\Omega M;\kk) \cong H_*(M;\kk)^!.$$
\end{enumerate}
In either situation, there is a twisting morphism
$$\kappa\colon H_*(M;\kk)\to H_*(\Omega M;\kk)$$
such that the twisted convolution $A_\infty$-algebra
$$\Hom^\kappa(H_*(M;\kk),H_*(\Omega M;\kk))$$
is quasi-isomorphic, as an $A_\infty$-algebra, to the Hochschild cochains on $C_*(\Omega M;\kk)$.

In particular, if $M$ is a $d$-dimensional manifold that is formal or coformal over $\kk$, then there is an isomorphism of graded algebras
$$H_{*+d}(LM;\kk) \cong H_* \Hom^\kappa(H_*(M;\kk),H_*(\Omega M;\kk)),$$
where the left hand side carries the Chas-Sullivan loop product.
\end{theorem}

This generalizes \cite[Theorem 1.2]{BerglundBorjeson}. As an illustration, we offer two case studies where the methods of \cite{BerglundBorjeson} do not apply, but where the new methods do apply. The first is an example of a formal but non-coformal manifold, $\CP{n}$. The Chas-Sullivan algebra of $\CP{n}$ was computed in \cite{CohenJonesYan}, but the methods here give a streamlined computation (in fact, the twisted convolution algebra model may be viewed as a chain-level refinement of the Cohen-Jones-Yan spectral sequence). The second example is a certain coformal but non-formal $7$-manifold $M$, obtained by pulling back the Hopf fibration $\eta\colon S^7\to S^4$ along the collapse map $S^2\times S^2 \to S^4$. We show that this manifold is coformal but not formal over $\ZZ$ and compute $H_{*+7}(LM;\ZZ)$. This calculation is new.


\section{Koszul $A_\infty$-algebras}
In this section we introduce a notion of Koszul $A_\infty$-algebra that extends Priddy's notion of a Koszul algebra \cite{Priddy}. First, let us recall the definition of $A_\infty$-algebras. We will follow the sign conventions of Lef\`evre-Hasegawa 
\cite{Lefevre-Hasegawa}.
\begin{definition}
Let $\kk$ be a commutative ring. An \emph{$A_\infty$-algebra} over $\kk$ is a graded $\kk$-module $A = \{A_i\}_{i\in \ZZ}$ together with maps
$$m_n\colon A^{\tensor n} \rightarrow A,\quad n\geq 1,$$
of degree $n-2$ such that
$$\sum_{\substack{r+s+t = n \\ u = r+1+t}} (-1)^{rs+t} m_u(1^{\tensor r}\tensor m_s\tensor 1^{\tensor t}) = 0$$
for every $n\geq 1$.
\end{definition}

Note that an $A_\infty$-algebra $A$ such that $m_n = 0$ for $n\geq 3$ is the same thing as a (non-unital) dg algebra. The differential $m_1$ of $A$ will also be denoted $d_A$.

\begin{definition}
\label{weight grading}
A \emph{weight grading} on an $A_\infty$-algebra $A$ is a decomposition of $A$ as a direct sum of graded $\kk$-modules,
$$A = \bigoplus_{k\in \ZZ} A(k),$$
such that $m_n\colon A^{\tensor n} \to A$ is homogeneous of weight $n-2$, in the sense that
$$m_n\big(A(i_1)\tensor \cdots \tensor A(i_n) \big) \subseteq A(i_1+\cdots + i_n + n - 2).$$
\end{definition}

Note that each component $A(k)$ is itself a graded $\kk$-module; in effect $A$ is bigraded. An element $x\in A(k)_i$ is said to have \emph{weight} $w(x) = k$ and \emph{(homological) degree} $|x| = i$.

Next, recall that the \emph{bar construction} of an $A_\infty$-algebra $A$ is the dg coalgebra $BA = \big(T^c(sA),b\big)$, where $T^c(sA)$ is the tensor coalgebra on $sA$ and $b$ is the differential given by $b = b_0 + b_1 + b_2 +\cdots$, where
$$b_{n-1}[sx_1|\ldots|sx_m] = \sum_{k=0}^{m-n} (-1)^{\epsilon_k} [sx_1|\ldots|sx_k|sm_n(x_{k+1},\ldots,x_{k+n})|\ldots|sx_m],$$
and where the sign is given by
$$\epsilon_k = 1+\sum_{i=1}^k |sx_i| + \sum_{j=1}^n (n-j)|sx_{k+j}|.$$
If $A$ is equipped with a weight grading, then the bar construction $BA$ admits a weight grading where
$$w [sx_1|\ldots|sx_n] = w(x_1) + \cdots + w(x_n)+n.$$
(The suspension operator $s$ increases weight and homological degree by $1$.)
The differential $b$ then becomes homogeneous of weight $-1$.

A weight grading on $A$ is called \emph{negative} if $A(k) = 0$ for $k\geq 0$. In this situation, the bar construction is concentrated non-positive weights,
$$BA(0) \xrightarrow{b} BA(-1) \xrightarrow{b} BA(-2) \to \cdots.$$
The weight zero homology of the bar construction,
$$A^\as = H_*(BA)(0) = \ker\big(b\colon BA(0) \to BA(-1)\big),$$
will be called the \emph{Koszul dual coalgebra of $A$}. It is a graded subcoalgebra of $BA$, with trivial differential.

\begin{definition} \label{def:koszul}
\begin{enumerate}
\item A \emph{Koszul weight grading} on an $A_\infty$-algebra $A$ is a negative weight grading such that the homology of the bar construction $BA$ is concentrated in weight $0$.

\item An $A_\infty$-algebra is called \emph{Koszul} if it admits a Koszul weight grading.
\end{enumerate}
\end{definition}

Thus, $A$ is Koszul if and only if the inclusion $A^\as \to BA$ is a quasi-isomorphism.

\begin{remark}
The above definition is modeled on the following well-known characterization of Koszul algebras: a (non-unital) quadratic algebra $A$ is Koszul if and only if the grading induced by the (negative) wordlength in the generators is a Koszul weight grading, see e.g.~\cite[Theorem 3.4.4]{LodayVallette}. In this case, the weight grading on $BA$ corresponds to the `syzygy degree' of \cite[\S3.3.1]{LodayVallette}. As we will see, the notion of a Koszul weight grading is however more flexible and applies not only to quadratic algebras.
\end{remark}

\begin{proposition} \label{prop:gc}
Let $C$ be a graded coalgebra with zero differential. Then the cobar construction $\Omega C$ is a Koszul dg algebra; the grading by tensor length,
$$\Omega C(-k) = (s^{-1}C)^{\tensor k},$$
is a Koszul weight grading.
\end{proposition}

\begin{proof}
We may view $C$ as a weight graded coalgebra by declaring it to be concentrated in weight $0$. The canonical quasi-isomorphism $C\to B\Omega C$ is weight homogeneous, showing the homology of $B\Omega C$ is concentrated in weight $0$.
\end{proof}

A chain complex $(A,m_1)$ of $\kk$-modules is called \emph{split} if there exists a contraction of chain complexes,
\begin{equation} \label{eq:contraction}
\bigSDR{(A,m_1)}{(H_*A,0),}{f}{g}{h}
\end{equation}
meaning $f$ and $g$ are chain maps such that $fg = 1$ and $h$ is a chain homotopy between $gf$ and $1$.
For instance, if $\kk$ is a PID then every chain complex of free $\kk$-modules with $\kk$-free homology is split. In particular, if $\kk$ is a field, then every chain complex is split. Recall that an $A_\infty$-algebra is called \emph{minimal} if $m_1=0$.

\begin{theorem} \label{thm:minimal}
Every Koszul $A_\infty$-algebra whose underlying chain complex is split is quasi-isomorphic to a minimal Koszul $A_\infty$-algebra.
\end{theorem}

\begin{proof}
Let $\big(A,\{m_n\}\big)$ be a Koszul $A_\infty$-algebra. Since the differential $m_1$ of $A$ is homogeneous of weight $-1$, the homology $H_*A$ may be equipped with a weight grading as follows:
$$H_*A(k) = \ker(A(k) \xrightarrow{m_1} A(k-1))/\im( A(k+1)\xrightarrow{m_1} A(k)).$$
Since $(A,m_1)$ is assumed to be split, there exists a contraction as in \eqref{eq:contraction}
and we may without loss of generality assume that $f$ and $g$ are homogeneous of weight $0$ and that $h$ is homogeneous of weight $1$.
The homotopy transfer theorem (see, e.g., \cite{Berglund2}) produces a minimal $A_\infty$-algebra structure $m' = \{m_n'\}_{n\geq 2}$ on $H_*A$ and an $A_\infty$-quasi-isomorphism
$$\{g_n\}_{n\geq 1}\colon \big(H_*A,\{m_n'\}\big) \to \big(A,\{m_n\}_{n\geq 1}\big)$$
with $g_1 = g$. The $A_\infty$-quasi-isomorphism $\{g_n\}$ corresponds to a quasi-isomorphism of dg coalgebras $G\colon B(A,m) \to B(H_*A,m')$. A glance at the explicit formulas for the transferred structure shows that $m_n'$ is homogeneous of weight $n-2$ and that $G$ is homogenous of weight $0$. Since $BA$ has homology concentrated in weight $0$, it follows that so does $BH_*A$. In other words, the weight grading on $(H_*A,m')$ is Koszul.
\end{proof}

\begin{corollary}
Let $C$ be a graded coalgebra. If $\Omega C$ is split as a chain complex, then the homology of the cobar construction $H_*(\Omega C)$ admits a minimal Koszul $A_\infty$-algebra structure, such that it is quasi-isomorphic to $\Omega C$.
\end{corollary}

\begin{proof}
Combine Theorem \ref{thm:minimal} and Proposition \ref{prop:gc}.
\end{proof}


In \cite[Corollary 2.10]{Berglund} (see also \cite[Remark 2.13]{BerglundBorjeson}), it was shown that a graded algebra $A$ is Koszul if and only if the bar construction $BA$ is formal as a dg coalgebra, giving an intrinsic characterization of the Koszul property for algebras. The next theorem extends this result to $A_\infty$-algebras. For simplicity we state the result when the ground ring $\kk$ is a field. 

\begin{theorem} \label{thm:koszul}
\label{koszulequivalencetheorem}
The following are equivalent for an $A_\infty$-algebra $A$ over a field $\kk$:
\begin{enumerate}
\item The $A_\infty$-algebra $A$ is quasi-isomorphic to a Koszul $A_\infty$-algebra.
\item The bar construction $BA$ is formal as a dg coalgebra.
\end{enumerate}
\end{theorem}

\begin{proof}
Suppose that $A$ is quasi-isomorphic to a Koszul $A_\infty$-algebra. Since $B$ preserves $A_\infty$-quasi-isomorphisms, we may without loss of generality assume that $A$ is Koszul itself. In this case, the inclusion $A^{\as} \to BA$ is a quasi-isomorphism of dg coalgebras. Since the coalgebra $A^{\as}$ has trivial differential, this shows that $BA$ is formal.

Conversely, if $BA$ is formal, then there exists a quasi-isomorphism of dg algebras $\Omega H_*(BA)\to \Omega BA$. By Proposition \ref{prop:gc}, the dg algebra $\Omega H_*(BA)$ is Koszul. Since $A$ is $A_\infty$-quasi-isomorphic to $\Omega BA$, this shows that $A$ is quasi-isomorphic to a Koszul $A_\infty$-algebra.
%
\end{proof}

\subsection{Koszul $A_\infty$-coalgebras}
Everything in the previous section can be dualized.

\begin{definition}
An $A_\infty$-coalgebra is a graded $\kk$-module $C = \{C_i\}_{i\in \ZZ}$ together with maps
$$\Delta_n\colon C \to C^{\tensor n},\quad n\geq 1,$$
of degree $n-2$ such that 
$$\sum_{\substack{r+s+t = n \\ u = r+1+t}} (-1)^{rs+t} (1^{\tensor r}\tensor \Delta_s\tensor 1^{\tensor t})\Delta_u = 0$$
for every $n\geq 1$.
\end{definition}

\begin{definition}
A \emph{weight grading} on an $A_\infty$-coalgebra $C$ is a decomposition of $C$ as a direct sum of graded $\kk$-modules,
$$C = \bigoplus_{k\in \ZZ} C(k),$$
such that $\Delta_n\colon C \to C^{\tensor n}$ is homogeneous of weight $n-2$, in the sense that
$$\Delta_n(C(k)) \subseteq \bigoplus C(i_1)\tensor \cdots \tensor C(i_n),$$
where the sum is over all $i_1,\ldots,i_n$ such that $i_1+\cdots +i_n = k + n - 2$.
\end{definition}

The cobar construction on an $A_\infty$-coalgebra $C$ is defined as $\Omega C = \big(T(s^{-1} C),\delta\big)$, where the differential is a sum of derivations $\delta = \delta_0+\delta_1+\cdots$ determined by
$$\delta_{n-1}(s^{-1} x) = (s^{-1})^{\tensor n} \Delta_n(x).$$
If $C$ is weight graded, then the cobar construction is weight graded by
$$w(s^{-1} x_1 \tensor \cdots \tensor s^{-1} x_n) = w(x_1) + \cdots + w(x_n) - n.$$
Then $\delta$ becomes homogeneous of weight $-1$.

If $C$ is positively weight graded, then $\Omega C$ is concentrated in non-negative weights,
$$\cdots \xrightarrow{\delta} \Omega C(2) \xrightarrow{\delta} \Omega C(1) \xrightarrow{\delta} \Omega C(0).$$
The weight zero homology of the cobar construction,
$$C^! = H_*(\Omega C)(0) = \coker\big(\delta\colon \Omega C(1) \to \Omega C(0)\big),$$
is called the \emph{Koszul dual algebra} of $C$. It is a quotient graded algebra of $\Omega C$, with trivial differential.

\begin{definition}
\begin{enumerate}
\item A Koszul weight grading on an $A_\infty$-coalgebra $C$ is a positive weight grading such that the homology of $\Omega C$ is concentrated in weight $0$.

\item An $A_\infty$-coalgebra is called \emph{Koszul} if it admits a Koszul weight grading.
\end{enumerate}
\end{definition}

Thus, $C$ is Koszul if and only if the map $\Omega C\to C^!$ is a quasi-isomorphism.

\begin{remark} \label{rem:presentation}
Note that $\Omega C(0)$ may be identified with the tensor algebra $T(s^{-1}C(1))$ and the image of $\delta\colon \Omega C(1) \to \Omega C(0)$ with the two-sided ideal generated by
$$\sum_{n\geq 1} (s^{-1})^{\tensor n} \Delta_n(x)$$  
for $x\in C(2)$. Thus, the Koszul dual algebra of $C$ admits a presentation where $C(1)$ enumerates the generators and $C(2)$ enumerates the relations. This presentation is quadratic if and only if $\Delta_n = 0$ for all $n\ne 2$, i.e., if $C$ is a graded coalgebra with trivial differential and higher operations.
\end{remark}

The results in the previous section have obvious duals. We state these results below for reference.

\begin{proposition} \label{prop:bar koszul}
Let $A$ be a graded algebra with zero differential. Then the bar construction $BA$ is a Koszul dg coalgebra; the grading by bar length,
$$BA(k) = (sA)^{\tensor k}$$
is a Koszul weight grading.
\end{proposition}

\begin{theorem}
Every Koszul $A_\infty$-coalgebra whose underlying chain complex is split is quasi-isomorphic to a minimal Koszul $A_\infty$-coalgebra.
\end{theorem}

\begin{corollary}
Let $A$ be a graded coalgebra. If $BA$ is split as a chain complex, then the homology of the bar construction $H_*(BA)$ admits a minimal $A_\infty$-coalgebra structure such that $H_*(BA)$ is Koszul and quasi-isomorphic to $BA$.
\end{corollary}

\begin{theorem}
The following are equivalent for an $A_\infty$-coalgebra $C$:
\begin{enumerate}
\item $C$ is quasi-isomorphic to a Koszul $A_\infty$-coalgebra.
\item The cobar construction $\Omega C$ is formal as a dg algebra.
\end{enumerate}
\end{theorem}

%

We now give some examples of Koszul $A_\infty$-(co)algebras.

\begin{example}
\begin{itemize}
\item As remarked earlier, every quadratic Koszul algebra is a Koszul $A_\infty$-algebra.

\item As shown above, the bar construction $BA$ of a graded algebra $A$ is a Koszul dg coalgebra. The presentation of the Koszul dual algebra $BA^!$ described in Remark \ref{rem:presentation} gives the `multiplication table' presentation of $A$; $BA^! = T(A)/(a\tensor b - a\cdot b| a,b\in A)$.

\item The Chevalley-Eilenberg complex $C_*(\gl) = (\Lambda^* \gl,d_{CE})$ of a Lie algebra $\gl$, with weight grading given by $C_*(\gl)(k) = \Lambda^k \gl$, is a Koszul dg coalgebra. The Koszul dual algebra $C_*(\gl)^!$ is isomorphic to the universal enveloping algebra $U\gl$, and the presentation from Remark \ref{rem:presentation} agrees with the standard presentation $U\gl = T(\gl)/(x\tensor y - y\tensor x - [x,y]| x,y\in \gl)$.

\item For a non-homogeneous Koszul algebra $A$, in the sense of Priddy, the co-Koszul complex (see \cite[\S4]{Priddy}) is a Koszul dg algebra, whose Koszul dual algebra is $A$.

\item In \cites{HeLu,DotsenkoVallette}, a notion of Koszul duality for associative algebras with relations generated by $R\subseteq V^{\otimes N}$ is discussed. The Koszul dual is defined there as an $A_\infty$-algebra with $m_2$ and $m_N$ as the only non-vanishing structure maps. With our definition of Koszul $A_\infty$-algebra, this $A_\infty$-algebra will be Koszul.
\end{itemize}

\end{example}

In later sections, we will see further non-trivial examples of Koszul $A_\infty$-algebras that are not equivalent to ordinary Koszul algebras.

\subsection{Twisting morphisms and twisted tensor products} \label{sec:twisted}
Let $C$ be a dg coalgebra and $A$ an $A_\infty$-algebra. The graded $\kk$-module $\Hom(C,A)$ admits an $A_\infty$-algebra structure with
$$\mu_1(f) = d_A\circ f - (-1)^{|f|} f\circ d_C,$$
$$\mu_n\big(f_1,\dots,f_n\big)=m_n\circ(f_1\otimes\dots\otimes f_n)\circ\Delta^{(n)}, \quad n\geq 2,$$
for $f,f_1,\ldots,f_n\in \Hom(C,A)$. Here $\{m_n\}$ denotes the $A_\infty$-structure on $A$ and $\Delta^{(n)}\colon C \to C^{\tensor n}$ the iterated coproduct on $C$. The graded vector space $\Hom(C,A)$ together with this $A_\infty$-structure will be referred to as the \emph{convolution $A_\infty$-algebra}.

Similarly, if $C$ is an $A_\infty$-coalgebra and $A$ is a dg algebra, then $\Hom(C,A)$ admits an $A_\infty$-algebra structure where
$$\mu_1(f) = d_A\circ f - (-1)^{|f|} f\circ d_C,$$
$$\mu_n\big(f_1,\dots,f_n\big)=m^{(n)}\circ(f_1\otimes\dots\otimes f_n)\circ\Delta_n, \quad n\geq 2,$$
where $\{\Delta_n\}$ is the $A_\infty$-coalgebra structure on $C$ and $m^{(n)}\colon A^{\tensor n} \to A$ denotes the iterated product on $A$.


\begin{definition} \label{def:twisting morphism}
A \emph{twisting morphism} is a map $\tau\colon C \to A$ of degree $-1$ such that
$$\sum_{n\geq 1} \mu_n(\tau,\dots,\tau)=0.$$
\end{definition}

For an $A_\infty$-algebra $A$, there is a twisting morphism $\tau_A\colon BA \to A$, called the \emph{universal twisting morphism}, given by the composite
$$BA \to sA \to A,$$
where the first map is the projection and the second the desuspension.

Similarly, for an $A_\infty$-coalgebra $C$, there is a twisting morphism $\tau_C \colon C\to \Omega C$, also called the universal twisting morphism, given by the composite
$$C\to s^{-1}C \to \Omega C,$$
where the first map is the desuspension and the second the inclusion of generators.

\begin{definition}
\begin{enumerate}
\item Let $A$ be a negatively weight graded $A_\infty$-algebra with Koszul dual coalgebra $A^\as$. The twisting morphism
$$\kappa_A\colon A^\as \to A$$
defined as the composite $A^\as \to BA \xrightarrow{\tau_A} A$ will be called the \emph{canonical twisting morphism} associated to $A$.
\item Let $C$ be a weight graded $A_\infty$-coalgebra with Koszul dual algebra $C^!$. The twisting morphism
$$\kappa_C \colon C \to C^!$$
defined as the composite $C\xrightarrow{\tau_C} \Omega C \to C^!$ will be called the \emph{canonical twisting morphism} associated to $C$.
\end{enumerate}
\end{definition}

\begin{definition}
Let $\tau\in \Hom(C,A)_{-1}$ be a twisting morphism. The \emph{twisted tensor product} $C\otimes_\tau A$ is the usual tensor product chain complex with the term $$\sum_{k\geq 2} (1 \otimes m_k)\circ(1\otimes \tau^{\otimes (k-1)}\otimes 1) \circ (\Delta^k \otimes 1)$$
added to the differential.
\end{definition}

\begin{remark}
In the following theorem and in our topological applications we need to consider all constructions with (co)units in the appropriate way. Thus, in the remainder of this article we will assume that all $A_\infty$-(co)algebras are considered with (co)units. We will also assume that (co)bar constructions, twisted tensor products, Hochschild cohomology and other related constructions are all (co)unital, even when not explicitly stated as such.
\end{remark}

\begin{theorem}
\label{thm:twistedtensorproductkoszul}
\begin{enumerate}
\item A negatively weight graded connected $A_\infty$-algebra $A$ is Koszul if and only if the twisted tensor product
$$A^\as\tensor_{\kappa_A} A$$
is contractible, where $\kappa_A\colon A^\as \to A$ is the canonical twisting morphism.

\item A positively weight graded connected $A_\infty$-coalgebra $C$ is Koszul if and only if the twisted tensor product
$$C\tensor_{\kappa_C} C^!$$
is contractible, where $\kappa_C\colon C\to C^!$ is the canonical twisting morphism.
\end{enumerate}
\end{theorem}

\begin{proof}
The theorem follows by an adaptation of the spectral sequence argument in \cite[\S2.5]{LodayVallette}. 
In the case of (1) this is a comparison of spectral sequences obtained from filtrations of $BA\otimes_{\tau_A}A$ and $A^\as\tensor_{\kappa_A} A$. These filtrations are coming from filtrations by homological degree on $BA$ and $A^\as.$

\end{proof}

\subsection{Applications to topological spaces}
Let $\kk$ be a field. Recall that a based topological space $X$ is called \emph{formal over $\kk$} if the singular chain complex $C_*(X;\kk)$ is quasi-isomorphic, as a dg coalgebra, to the homology coalgebra $H_*(X;\kk)$. Dually, the space $X$ is called \emph{coformal over $\kk$} if the singular chains on the based loop space $C_*(\Omega X;\kk)$ is quasi-isomorphic, as a dg algebra, to the homology algebra $H_*(\Omega X;\kk)$. By applying the homotopy transfer theorem to the dg coalgebra $C_*(X;\kk)$ one obtains a minimal $A_\infty$-coalgebra structure on $H_*(X;\kk)$, where the binary coproduct is the ordinary coproduct in homology. Similarly, the homology $H_*(\Omega X;\kk)$ is endowed with a minimal $A_\infty$-algebra structure where $m_2$ is the Pontryagin product. Theorem \ref{thm:koszul} allows us to interpret formality and coformality in terms of Koszulness of these $A_\infty$-structures.

\begin{theorem}
\label{thm:formalkoszul}
Let $X$ be a path connected based topological space and let $\kk$ be a field. The following are equivalent:
\begin{enumerate}
\item The space $X$ is formal over $\kk$.
\item $H_*(\Omega X;\kk)$ is a Koszul $A_\infty$-algebra.
\end{enumerate}
In this situation, the homology coalgebra $H_*(X;\kk)$ is isomorphic to the Koszul dual coalgebra of $H_*(\Omega X;\kk)$.
\end{theorem}

\begin{proof}
This follows from Theorem \ref{koszulequivalencetheorem} and the fact that $BC_*(\Omega X;\kk)$ is quasi-isomorphic to $C_*(X;\kk)$ as a dg coalgebra, see 
Theorem 6.3 of \cite{FelixHalperinThomas}.
\end{proof}

\begin{remark}
When $X$ is formal over $\kk$, the Koszul weight grading on $H_*(\Omega X;\kk)$ corresponds to the `lower gradation' of the non-commutative bigraded minimal model for the cohomology ring $H^*(X;\kk)$, cf.
~\cite{HalperinStasheff}.
\end{remark}

Dually, we have the following theorem.
\begin{theorem}
\label{thm:coformalkoszul}
Let $X$ be a simply connected space and let $\kk$ be a field. The following are equivalent:
\begin{enumerate}
\item The space $X$ is coformal over $\kk$.
\item $H_*(X;\kk)$ is a Koszul $A_\infty$-coalgebra.
\end{enumerate}
In this situation, the Pontryagin algebra $H_*(\Omega X;\kk)$ is isomorphic to the Koszul dual algebra of $H_*(X;\kk)$.
\end{theorem}

\begin{proof}
This follows from the obvious dual version of Theorem \ref{koszulequivalencetheorem} together with the well-known fact that $\Omega C_*(X;\kk)$ is quasi-isomorphic to $C_*(\Omega X;\kk)$ as a dg algebra.
\end{proof}

\begin{remark}
\label{remark:PID}
The results in this section can be extended to the case when $\kk$ is a PID if one assumes that $H_*(X;\kk)$ and $H_*(\Omega X;\kk)$ are free $\kk$-modules. 
\end{remark}

\section{Hochschild cohomology and obstructions to formality}
In this section, we will use the notion of Koszulness for $A_\infty$-algebras to write down small chain complexes for computing Hochschild cohomology, generalizing the results of 
\cite{BerglundBorjeson}. We also discuss weight gradings on Hochschild cohomology and obstructions for formality and coformality.

\subsection{Hochschild cochains and twisted convolution algebras}

Given a convolution algebra $\Hom(C,A),$ and a twisting morphism $\tau:C\rightarrow A,$ we can define a new $A_\infty$-structure $\{\mu_n^\tau\}$ on $\Hom(C,A)$ by
$$\mu_n^\tau(f_1,\ldots,f_n) = \sum_{i\geq 0} \mu_{n+i}((\tau^{\tensor i}*f_1\tensor\ldots\tensor f_n)),$$
where $*$ denotes the anti-symmetric shuffle product, and where $\{\mu_n\}$ is the convolution $A_\infty$-algebra structure on $\Hom(C,A)$ described in \S\ref{sec:twisted}.
The first maps are given by, 
\begin{align*}
\mu_1^\tau(f) = & \mu_1(f) +  \mu_2(\tau,f) + (-1)^{|sf|}\mu_2(f,\tau) \\
&  + \mu_3(\tau,\tau,f) +(-1)^{|sf|}\mu_3(\tau,f,\tau)  + \mu_3(f,\tau,\tau) + \cdots,
\end{align*}
and
\begin{align*}
\mu_2^\tau(f,g) = & \mu_2(f,g) + \mu_3(\tau,f,g) + (-1)^{|sf|}\mu_3(f,\tau,g) + (-1)^{|sf|+|sg|}\mu_3(f,g,\tau) \\
& + \mu_4(\tau,\tau,f,g) + (-1)^{|sf|}\mu_4(\tau,f,\tau,g) + \cdots.
\end{align*}

The Hochschild cohomology complex can be defined as a convolution algebra twisted by the universal twisting morphism as follows. 

\begin{definition}
Let $A$ be a weight graded $A_\infty$-algebra. The Hochschild cohomology complex $C^*(A,A)$ is the weight graded $A_\infty$-algebra defined by $\Hom^{\tau_A}(BA,A),$ the convolution algebra twisted with the universal twisting morphism $\tau_A.$ 
Dually, the Hochschild cohomology complex $C^*(C,C)$ of a weight graded $A_\infty$-coalgebra $C$ is defined as the weight graded $A_\infty$-algebra $\Hom^{\tau_C}(C,\Omega C).$ 
\end{definition}

This point of view enables us, in the case of Koszul $A_\infty$-(co)algebras, to construct twisted convolution algebras that are smaller than the Hochschild cohomology complex but have the same homology.

\begin{theorem} \label{thm:koszulhochschild}
\begin{enumerate}
\item Let $A$ be a Koszul $A_\infty$-algebra with Koszul dual coalgebra $A^\as$ and canonical twisting morphism $\kappa_A\colon A^\as \to A$. There are quasi-isomorphisms of weight graded $A_\infty$-algebras $$\Hom^{\kappa_A}(A^\as,A)\sim C^*(A,A)\sim C^*(A^\as,A^\as).$$

\item Let $C$ be a Koszul $A_\infty$-coalgebra with Koszul dual algebra $C^!$ and let $\kappa_C\colon C\to C^!$ be the canonical twisting morphism. There are quasi-isomorphisms of weight graded $A_\infty$-algebras $$\Hom^{\kappa_C}(C,C^!)\sim C^*(C,C)\sim C^*(C^!,C^!).$$
\end{enumerate}
\end{theorem}

\begin{proof}
We will prove the first statement, the proof of the second one is analoguous.

The proof relies on the basic perturbation lemma and we assume that the reader is familiar with it, for an introduction, see e.g. \cite{Berglund}.
Consider the injective quasi-isomorphism $f: A^\as \to BA$ that exists since $A$ is a Koszul $A_\infty$-algebra. It is possible to extend $f$ to a contraction of chain complexes,
$$\bigSDR{\big(BA,d_{BA}\big)}{\big(A^\as,0\big)}{g}{f}{h}.$$
We apply the dg-functor $\Hom(-,A)$ and obtain a new contraction.
$$\bigbigSDR{\big(\Hom(BA,A),\partial\big)}{\big(\Hom(A^\as,A),0\big)}{f^*}{g^*}{h^*}$$
Since $f$ is a morphism of coalgebras, one sees that $f^*$ is a strict morphism of $A_\infty$-algebras. 
Consider the initiator $$t=\sum_{k\geq 1} \mu_{k+1}((\tau_A,\dots,\tau_A)*(-))$$ where $\tau_A\colon BA\to A$ is the universal twisting morphism.
This choice means that $\Hom(BA,A)$ perturbed with $t$ is isomorphic to $\Hom^{\tau_A}(BA,A).$
If we apply the basic perturbation lemma with $t$ as initiator, we obtain a new contraction
$$\bigbigSDR{\big(\Hom(BA,A),\partial + t\big)}{\big(\Hom(A^\as,A),t'\big)}{f'}{g'}{h'}.$$
To see that the sum $\sum_{n\geq 0}(h^*t)^n$ converges, we use the fact that $A$ carries a weight-grading. The chain complex $\Hom(BA,A)$ inherits a filtration from this grading, where $t$ decreases 
the filtration degree and $h^*$ preserves it. It follows that $\sum_{n\geq 0}(h^*t)^n$ converges point-wise. Since $f^*$ is a strict $A_\infty$-morphism, we can simplify the formulas for $f'$ and $t'.$ 
Indeed, $f'$ is given by $f' = f^*+f^*th^*+f^*th^*th^*+\dots$. We note that $$f^*th^*=f^*\sum_{k\geq 1} \mu_{k+1}((\tau_A,\dots,\tau_A)*(h^*))=\sum_{k\geq 1} \mu_{k+1}((f^*\tau_A,\dots,f^*\tau_A)*(f^*h^*))=0,$$ since $f^*$ is a strict morphism and $f^*h^*=0.$ Thus, we see that $f'=f^*$. 
Similarly, $t'=f^*tg^*+f^*th^*tg^*+\dots$, where $$f^*tg^*=f^*\sum_{k\geq 1} \mu_{k+1}((\tau_A,\dots,\tau_A)*(g^*))=\sum_{k\geq 1} \mu_{k+1}((f^*\tau_A,\dots,f^*\tau_A)*(f^*g^*))=$$ $$=\sum_{k\geq 1} \mu_{k+1}((\kappa_A,\dots,\kappa_A)*(-)),$$ which is the differential on $\Hom^{\kappa_A}(A^\as,A).$ The higher terms all vanish in the same way as above, so we may identify $\big(\Hom(A^\as,A),t'\big)$ with $\Hom^{\kappa_A}(A^\as,A).$ Thus, we see that $f'=f^*$ is a strict $A_\infty$-quasi-isomorphism $\Hom^{\tau_A}(BA,A)\rightarrow \Hom^{\kappa_A}(A^\as,A).$ 

Thus, we have proved the first part of (1), that $\Hom^{\kappa_A}(A^\as,A)\sim C^*(A,A)$ and we only need to prove that $\Hom^{\kappa_A}(A^\as,A)\sim C^*(A^\as,A^\as).$ This can be done via a spectral sequence argument. The key is to consider the following inductively defined filtration of $A^\as.$ 
$$F_0(A^\as)=\kk1$$
$$F_r(A^\as)= \{x\in A^\as | \Delta(x) - 1\otimes x - x \otimes 1 \in F_{r-1}(A^\as)\otimes F_{r-1}(A^\as) \}$$
By analyzing the resulting spectral sequence we see that we indeed obtain a quasi-isomorphism.

\end{proof}

\subsection{Obstructions to formality}


\begin{theorem}
Suppose that $A$ is a dg algebra over $\kk$ such that $A$ is split as a chain complex of $\kk$-modules.
There is a sequence of `obstruction classes'
$$[m_k] \in HH^2(H_*A,H_*A)(-k),\quad k\geq 3,$$
where $[m_k]$ is defined if the previous classes $[m_3],\ldots,[m_{k-1}]$ vanish. If $[m_k] =0$ all $k\geq 3$, then the dg algebra $A$ is formal.
\end{theorem}

\begin{proof}
Since $A$ is assumed split, we have a retract
$$\bigSDR{(A,m_1)}{(H_*A,0)}{f}{g}{h}.$$
Using the homotopy transfer theorem we obtain a minimal $A_\infty$-algebra structure on the homology, i.e., a sequence of maps
$$m_k\colon (H_*A)^{\otimes k}\rightarrow H_*A, \quad k\geq 2,$$
of degree $k-2$, and where $m_2$ is the usual product on homology. Considering $H_*A$ concentrated in weight $0,$ we may interpret $m_k$ as a Hochschild cochain of (total) cohomological degree $2$ and weight $-k$.
 If $[m_k]=0,$ for all $k\geq 3,$ the algebra $A$ is formal over $\kk.$

For more details about obstructions for formality via Hochschild cohomology, we refer to \cite{Saleh} and the references therein. 
In loc.\ cit.\ the results are stated in a more general operadic setting but over a field of characteristic $0$. For the associative operad we can relax this as long as we assume that the relevant modules are free.
\end{proof}

The following useful proposition will allow us to deduce integral formality from rational formality in favorable situations.

\begin{proposition} \label{prop:integral formality}
Let $A$ be a dg algebra over $\ZZ$ such that $A$ and $H_*A$ are free as $\ZZ$-modules with $H_*(A)$ considered in weight $0.$ Suppose that the obstruction group $HH^2(H_*A,H_*A)(-k)$ is torsion free for all $k\geq 3$. Then $A$ is formal over $\ZZ$ if and only if $A\tensor \QQ$ is formal over $\QQ$.
\end{proposition}

\begin{proof}
Write $U = H_*A$ for brevity. The claim follows from the easily checked fact that the obstruction classes for $A$ map to the obstruction classes for $A\tensor \QQ$ under the canonical map
$$HH^2(U,U)(-k) \to HH^2(U,U)(-k)\tensor \QQ \cong HH^2(U\tensor \QQ,U\tensor \QQ)(-k).$$
Clearly, this map is injective if $HH^2(U,U)(-k)$ is torsion free.
\end{proof}

\section{Applications to free loop space homology: two case studies}
In this section we will see the theory developed in the previous sections in action. We will work over the ring $\ZZ$ of integers.

We offer two case studies. Firstly, we will treat complex projective space as an example of a manifold that is formal but not coformal over $\ZZ$. The result of the calculation is not new, but the perspective is, and the reader may find it interesting to compare our approach to the existing ones, such as \cite{CohenJonesYan}.

Secondly, we treat a certain $7$-manifold, as an example of a manifold that is coformal but not formal over $\ZZ$. This is a new computation. 
Our methods apply more generally to other coformal but not formal manifolds, but in the interest of brevity and clarity we choose to focus on a specific example.

\subsection{Free loop space homology through Hochschild cohomology}

It is well known that Hochschild cohomology can be used to compute the free loop space homology of a manifold.

\begin{theorem}
\label{stringtopologyhochschildtheorem}
Let $M$ be a simply connected manifold of dimension $n$ and let $\kk$ be a commutative ring. There are graded ring isomorphisms $$H_{*+n}(LM;\kk)\cong HH^*(C^*(M;\kk),C^*(M;\kk))\cong HH^*(C_*(\Omega M;\kk),C_*(\Omega M;\kk))$$ where the algebra structure on the left hand side is the Chas-Sullivan loop product.
\end{theorem}

\begin{proof}
For the first isomorphism, see \cite{CohenJones}. For the second isomorphism, see \cite{Malm}. 
\end{proof}

We can now state the main theorem used for computing our free loop spaces.

\begin{theorem}
\label{stringtopologyformaltheorem}
Let $M$ be a simply connected closed $n$-dimensional manifold. Let $\kk$ be a PID such that $H_*(M;\kk)$ and $H_*(\Omega M;\kk)$ are free $\kk$-modules.
\begin{enumerate}
\item If $M$ is formal over $\kk$, then there is an algebra isomorphism
$$H_{*+n}(LM)\cong H_*\Hom^{\kappa_{H_*\Omega M}}(H_*M,H_*\Omega M),$$
where $H_*(\Omega M;\kk)$ is considered as $A_\infty$-algebra.
\item If $M$ is coformal, there is an algebra isomorphism $$H_{*+n}(LM)\cong H_*\Hom^{\kappa_{H_*M}}(H_*M,H_*\Omega M),$$ where $H_*M$ is considered as $A_\infty$-coalgebra.
\end{enumerate}
\end{theorem}

\begin{proof}
This follows from Theorems \ref{thm:formalkoszul}, \ref{thm:coformalkoszul}, \ref{thm:koszulhochschild} and \ref{stringtopologyhochschildtheorem} together with Remark \ref{remark:PID}.
\end{proof}

\subsection{Complex projective space}
To compute the free loop space homology of $\CP{n}$ with coefficients in $\mathbb{Z}$ we will use the following approach. First we introduce an $A_\infty$-algebra $A,$ prove that it is Koszul and note that its Koszul dual coalgebra $A^\as$ is isomorphic to $H_*(\CP{n};\mathbb{Z}).$ This enables us to use Theorem \ref{thm:koszulhochschild} to compute the Hochschild cohomology of $H_*(\CP{n})$ with its weight grading. From this we can use the obstruction theory of Proposition \ref{prop:integral formality} together with the familiar fact that $\CP{n}$ is formal over $\mathbb{Q}$ to prove that $\CP{n}$ is formal over $\mathbb{Z}.$ Then finally we can apply Theorem \ref{stringtopologyformaltheorem} to see that our Hochschild cohomology computation also calculates string topology.

\begin{theorem}
\label{complexprojectivecomputation}
The manifold $\CP{n}$ is formal over $\mathbb{Z}.$ Moreover, the Hochschild cohomology algebra of $H^*(\CP{n};\mathbb{Z})$ is isomorphic to
$$\frac{\Lambda[x,y,z]}{(x^{n+1},(n+1)x^nz,x^ny)},$$
where $\Lambda[x,y,z]$ is the free graded commutative algebra on generators $x,y,z$ with degrees $|x|=-2, |y|=-1$ and $|z|=2n.$
\end{theorem}

\begin{proof}
Let $A$ denote the $A_\infty$-algebra given by the free graded commutative algebra $\Lambda(\alpha,\beta),$ with $|\alpha|=1,|\beta|=2n$, together with the higher $A_\infty$-structure maps \begin{align*}
m_{n+1}(\alpha,\dots,\alpha)&=\beta, \\
m_{n+1}(\dots,\beta\phi,\dots)&=\beta m_{n+1}(\dots,\phi,\dots), \\
m_{n+1}(\dots,1,\dots)&=0.
\end{align*}
We put $m_k=0$ if $k$ is not equal to $2$ or $n+1.$
We give $A$ the weight grading determined by $w(\alpha)=-1$ and $w(\beta)=-2.$
From inspection of $BA$ we see that $A^\as= \mathbb{Z}\{1,x_1,\dots,x_n\}\cong H_*(\CP{n};\mathbb{Z})$ where $|x_i|=2i.$
There is a twisting morphism $\kappa_A:A^\as\rightarrow A$ taking $x_1$ to $\alpha.$ We want to show that $A^\as\tensor_{\kappa_A} A$ is contractible so that we can apply Theorem \ref{thm:twistedtensorproductkoszul} to conclude that $A$ is Koszul.

The differentials are given as follows on the basis elements where for simplicity we denote $x_0=1.$
\begin{equation*}
d_\tau(x_i\otimes \beta^k) = 
\begin{cases} 
x_{i-1}\otimes \alpha \beta^k & \text{if } i\geq 1 \\
0 & \text{if } i= 0 \\
\end{cases}
\end{equation*}
\begin{equation*}
d_\tau(x_i\otimes \alpha\beta^k) = 
\begin{cases} 
x_0\otimes \beta^{k+1} & \text{if } i = n \\
0 & \text{if } i< n \\
\end{cases}
\end{equation*}
We see that the basis elements pairs up except for $1\otimes 1$ showing that the complex is indeed contractible. 
Now we can apply Theorem \ref{thm:koszulhochschild} to calculate Hochschild cohomology of $A^\as.$
For ease of writing we will dualize $A^\as$ and use $A^!\cong H^*(\CP{n};\mathbb{Z}),$ the linear dual of $A^\as.$ 

Since it is isomorphic to $T(a)/(a^{n+1})$ with $|a|=-2,$ $\Hom^{\kappa_A}(A^\as,A)$ is isomorphic to the $A_\infty$-algebra $A^!\otimes A$ twisted by the element $a\otimes \alpha.$ The twisted differential is given on generators as follows.


\begin{equation*}
\partial_\tau(a^\ell \otimes \beta^k) = 0, \quad \quad\\
\partial_\tau(a^\ell \otimes \alpha\beta^k)=
\begin{cases} 
(n+1) a^n \otimes \beta^{k+1} & \text{if } \ell = 0 \\
0 & \text{if } \ell > 0 \\
\end{cases}
\end{equation*}

The twisted multiplication is given by $(a^k\otimes\alpha^i\beta^\ell)(a^p\otimes\alpha^j\beta^q)=$ $$a^{k+p}\otimes\alpha^{i+j}\beta^{\ell+q}+\sum_{antisymmetric \atop shuffles} \pm a^{k+p+n-1}\otimes m_{n+1}(\alpha,\dots,\alpha^i,\dots,\alpha^j,\dots,\alpha)\beta^{\ell+q}.$$
The first term is zero unless $i+j < 2$ and the second term is zero unless $i+j=2,$ so we only have one term at the time for any two basis elements. Moreover, the second term is zero unless $k+p <2$ so the only possible non-zero second terms are given by the following.
\begin{align*}
(1\otimes \alpha \beta^k)(1\otimes \alpha \beta^q)&=\binom{n+1}{2}a^{n-1}\otimes \beta^{k+q+1} \\
(a\otimes \alpha \beta^k)(1\otimes \alpha \beta^q)&=\binom{n+1}{2}a^{n}\otimes \beta^{k+q+1} \\
(1\otimes \alpha \beta^k)(a\otimes \alpha \beta^q)&=\binom{n+1}{2}a^{n}\otimes \beta^{k+q+1} 
\end{align*}
Since $1\otimes \alpha \beta^k$ is not a cycle, these terms will not affect the multiplication in the homology.
We see that $a\otimes 1, a\otimes \alpha$ and $1\otimes \beta$ are algebra generators for the homology and if we write $a\otimes 1=x, a\otimes \alpha=y$ and $1\otimes \beta=z$ we see that we get the relations $x^{n+1}=0,x^ny=0$ and $(n+1)x^nz=0.$
From this description of the Hochschild cohomology of $A^\as$ we see that we can apply Proposition \ref{prop:integral formality} together with the fact that $\CP{n}$ is rationally formal (since it is a K\"ahler manifold) to conclude that $\CP{n}$ is formal over the integers.
\end{proof}

\begin{corollary}
There is an isomorphism $$H_{*+2n}L\CP{n}\cong \frac{\Lambda[x,y,z]}{(x^{n+1},(n+1)x^nz,x^ny)},$$ of graded algebras where the algebra structure on the left hand side is given by the string topology multiplication.
\end{corollary}

\begin{proof}
This follows from the above theorem together with Theorem \ref{stringtopologyformaltheorem}.
\end{proof}

\subsection{A non-formal $7$-manifold}
The manifold $M$ is defined as the pullback of the Hopf fibration $\eta\colon S^7 \to S^4$ along the collapse map $S^2\times S^2 \to S^4$;
$$
\xymatrix{M \ar[d] \ar[r] & S^7 \ar[d]^-\eta \\ S^2 \times S^2 \ar[r]^-\wedge & S^4.}
$$
In other words,
$$M = \set{(x,y,z)\in S^2\times S^2 \times S^7}{x\wedge y = \eta(z)} \subset S^2 \times S^2 \times S^7.$$
The manifold $M$ is the total space of a principal $S^3$-bundle,
$$S^3 \xrightarrow{i} M \xrightarrow{p} S^2\times S^2.$$

We begin by computing the cohomology ring of $M$.
\begin{theorem}
The cohomology ring of $M$ is given by
$$H^*M = \ZZ \oplus \ZZ a \oplus \ZZ b \oplus \ZZ x \oplus \ZZ y \oplus \ZZ M,$$
where $|a| = |b| =2$, $|x|=|y| = 5$, and $|M|= 7$.
The ring structure is determined by $a\smile x = - b\smile y = M$; all other non-trivial products are zero.
\end{theorem}

\begin{proof}
It is well known that the Euler class of the $S^3$-bundle $S^7\rightarrow S^4$ is $\pm 1$ times the fundamental class. We can also see this by applying the Gysin sequence associated to $\eta \colon S^7\rightarrow S^4.$ Euler classes are preserved under pullbacks, so the Euler class of the $S^3$-bundle $M\rightarrow S^2\times S^2$ is $\pm 1$ times the fundamental class of $S^2\times S^2.$ Now we can analyze the associated Gysin sequence. We see that $H^4(M)\cong 0$ since $$H^0(S^2\times S^2)\xrightarrow{\smile e} H^4(S^2\times S^2)\rightarrow H^4(M)\rightarrow 0$$ has to be exact and the first map is multiplication with the Euler class and thus an isomorphism. We also see that $H^5(M)\cong \mathbb{Z}^2, H^6(M)\cong 0$ and $H^7(M)\cong \mathbb{Z}$ from the Gysin sequence. Now the result follows from Poincar\'{e} duality if we choose generators $x,y$ for $H^5(M)$ and let $a,b\in H^2(M)$ be defined via the Poincar\'{e} duality pairing such that $a\smile x=-b\smile y=M.$
\end{proof}

Next, we turn to the homotopy groups.

\begin{theorem} \label{thm:homotopy groups}
The manifold $M$ is simply connected and for every $k\geq 2$, the map
$$\psi\colon \pi_k(S^2) \oplus \pi_k(S^2) \oplus \pi_k(S^3) \to \pi_k(M),$$
sending $(a,b,c)$ to $\alpha\circ a + \beta \circ b + [\alpha,\beta] \circ c$, is an isomorphism.
\end{theorem}

\begin{proof}
There are two embeddings $\alpha,\beta\colon S^2\to M$ given by $\alpha(x) = (x,*,*)$ and $\beta(y) = (*,y,*)$. The composite
$$S^2\vee S^2 \xrightarrow{\alpha\vee \beta} M \to S^2\times S^2$$
is the standard inclusion of the wedge into the product. In particular, the composite map $\pi_k(S^2\vee S^2) \to \pi_k(M) \to \pi_k(S^2\times S^2)$ is split surjective. It follows that so is $\pi_k(M) \to \pi_k(S^2\times S^2)$. From the long exact homotopy sequence of the fibration $S^3 \xrightarrow{i} M \xrightarrow{p} S^2\times S^2$ we deduce that the map
$$\pi_k(S^2) \oplus \pi_k(S^2) \oplus \pi_k(S^3) \to \pi_k(M)$$
sending $(a,b,c)$ to $\alpha\circ a + \beta \circ b + i \circ c$, is an isomorphism for every $k\geq 2$. It remains to identify the inclusion of the fiber, $i\colon S^3\to M$, with the Whitehead product $[\alpha,\beta]$ up to homotopy.

The universal Whitehead product $w_{2,2}\colon S^3 \to S^2\vee S^2$ is null when composed with the inclusion into the product $S^2\times S^2$. It follows that the composite map $S^3 \to S^2\vee S^2 \to M$ factors over the fiber of $p\colon M\to S^2\times S^2$, giving a self-map $\lambda$ of $S^3$ such that the diagram
$$
\xymatrix{S^3 \ar@{-->}[d]^-\lambda \ar[r]^-{w_{2,2}} & S^2\vee S^2 \ar[d]^-{\alpha\vee \beta} \ar[r] & S^2\times S^2 \ar@{=}[d] \\ S^3 \ar[r]^-i & M \ar[r]^-p & S^2\times S^2}
$$
commutes up to homotopy. By looking at the induced maps on $\pi_3$, we get a commutative diagram
$$
\xymatrix{0 \ar[r] & \pi_3(S^3) \ar[r] \ar[d]^-{\lambda_*} & \pi_3(S^2\vee S^2) \ar[d] \ar[r] & \pi_3(S^2\times S^2) \ar@{=}[d] \ar[r] & 0 \\
0 \ar[r] & \pi_3(S^3) \ar[r]^-{i_*} & \pi_3(M) \ar[r]^-{p_*} & \pi_3(S^2 \times S^2) \ar[r] & 0.}
$$
The upper row is split exact; this follows from the Hilton-Milnor theorem. The bottom row is split exact by our considerations above.

By comparing homology, the map $\alpha\vee \beta\colon S^2\vee S^2 \to M$ is seen to be $4$-connected. In particular, the middle vertical map in the diagram above is an isomorphism. It follows from the five lemma that $\lambda_*$ is an isomorphism, in other words $\lambda\colon S^3 \to S^3$ has degree $\pm 1$. Since $i\circ \lambda \simeq (\alpha \vee \beta)\circ w_{2,2} = [\alpha,\beta]$, this implies that $i$ is homotopic to $\pm [\alpha,\beta]$.
\end{proof}

The following corollary will be useful when we later compute the Pontryagin ring $H_*(\Omega M;\ZZ)$.

\begin{corollary} \label{cor:kernel}
The map $(\alpha\vee \beta)_* \colon \pi_k(S^2\vee S^2) \to \pi_k(M)$ is an isomorphism for $k\leq 3$ and split surjective for all $k\geq 4$.
The kernel of $\pi_4(S^2\vee S^2) \to \pi_4(M)$ is isomorphic to $\ZZ^2$, generated by the Whitehead products $[[\iota_1,\iota_2],\iota_1]$ and $[[\iota_1,\iota_2],\iota_2]$.
\end{corollary}

\begin{proof}
We have already established that the map $\alpha\vee \beta\colon S^2\vee S^2 \to M$ is $4$-connected. Let $\iota_1,\iota_2\colon S^2\to S^2\vee S^2$ denote the canonical inclusion maps. Define $\varphi\colon  \pi_k(S^2) \oplus \pi_k(S^2) \oplus \pi_k(S^3) \to \pi_k(S^2\vee S^2)$ by
$$\varphi(a,b,c) = \iota_1\circ a + \iota_2 \circ b + [\iota_1,\iota_2] \circ c.$$
The composite $\varphi\colon  \pi_k(S^2) \oplus \pi_k(S^2) \oplus \pi_k(S^3) \xrightarrow{\varphi} \pi_k(S^2\vee S^2) \xrightarrow{(\alpha\vee \beta)_*} \pi_k(M)$ is equal to the isomorphism $\psi$. It follows that $(\alpha\vee \beta)_*$ is split onto, as claimed.

By the Hilton-Milnor theorem, the map
$$\xi\colon \pi_4(S^2) \oplus \pi_4(S^2) \oplus \pi_4(S^3) \oplus \pi_4(S^4) \oplus \pi_4(S^4) \to \pi_4(S^2\vee S^2),$$
$$\xi(a,b,c,d,e) = \iota_1\circ a + \iota_2\circ b + [\iota_1,\iota_2]\circ c + [[\iota_1,\iota_2],\iota_1] \circ d + [[\iota_1,\iota_2],\iota_2] \circ e,$$
is an isomorphism. In view of Theorem \ref{thm:homotopy groups}, the composite of $\xi$ with $(\alpha\vee \beta)_*\colon \pi_4(S^2\vee S^2) \to \pi_4(M)$ clearly has kernel $\pi_4(S^4) \oplus \pi_4(S^4)$ generated by $[[\iota_1,\iota_2],\iota_1]$ and $[[\iota_1,\iota_2],\iota_2]$.
\end{proof}

Another corollary is a description of the rational homotopy Lie algebra and, as a consequence, the rational Pontryagin ring $H_*(\Omega M;\QQ)$.

\begin{corollary}
The rational homotopy Lie algebra of the manifold $M$ is $5$-dimensional with basis $\alpha,\beta,\alpha^2,\beta^2,[\alpha,\beta]$. In particular, $M$ is rationally elliptic.

A presentation is given by
$$\pi_*(\Omega M) \tensor \QQ = \LL(\alpha,\beta)/([[\alpha,\beta],\alpha],[[\alpha,\beta],\beta]).$$
Hence, the rational loop space homology algebra is given by
$$H_*(\Omega M;\QQ) \cong \QQ\langle \alpha,\beta\rangle / ([[\alpha,\beta],\alpha],[[\alpha,\beta],\beta]).$$
The Poincar\'e series of the loop space is given by
$$\sum_{k\geq 0} \rank(H_k \Omega M) t^k = \frac{1}{(1-t)^2(1-t^2)}.$$
\end{corollary}

\begin{proof}
The description of the rational homotopy Lie algebra follows immediately from Theorem \ref{thm:homotopy groups}; recall that $\pi_*(S^2)\tensor \QQ$ has basis $\iota$, $\eta$, with $\iota^2 =\frac{1}{2}[\iota,\iota] = \eta$, and $\pi_*(S^3)\tensor \QQ$ has basis $\iota$.

The description of the rational loop space homology algebra follows from the Milnor-Moore theorem: since $M$ is simply connected, the loop space homology algebra $H_*(\Omega M;\QQ)$ is isomorphic to the universal enveloping algebra $UL$ of the graded Lie algebra $L = \pi_*(\Omega M)\tensor \QQ$.

By the Poincar\'e-Birkhoff-Witt theorem, there is an isomorphism of graded vector spaces $UL \cong \Lambda L$, where $\Lambda L$ denotes the free graded commutative algebra on $L$. The description of the Poincar\'e series of $H_*(\Omega M;\QQ)$ is an easy consequence of this fact.
\end{proof}

We now turn to the computation of the integral Pontryagin ring. First, we need to establish an auxiliary result.

\begin{lemma}
The Serre spectral sequence of the fibration
$$\Omega S^3 \to \Omega M \to \Omega (S^2\times S^2)$$
collapses at the $E_2$-page. Hence, there is a filtration of the Pontryagin ring $H_*\Omega M$ such that the associated graded ring is isomorphic to  $$H_*\Omega (S^2 \times S^2) \tensor H_*\Omega S^3 \cong \ZZ\langle \alpha,\beta,\gamma\rangle/([\alpha,\beta], [\alpha,\gamma],[\beta,\gamma]).$$
In particular, $H_*\Omega M$ is torsion-free.
\end{lemma}

\begin{proof}
First, we note that $H_*\Omega(S^2\times S^2)$ is isomorphic to $\ZZ\langle\alpha,\beta\rangle/(\alpha\beta + \beta\alpha)$ and $H_*\Omega S^3$ is isomorphic to the polynomial ring $\ZZ[\gamma]$, where $\alpha,\beta\in H_1\Omega(S^2\times S^2)$ and $\gamma\in H_2\Omega S^3$ are the images under the Hurewicz homomorphism of generators for $\pi_1\Omega (S^2\times S^2)\cong \ZZ^2$ and $\pi_2\Omega S^3\cong \ZZ$, respectively.

The Serre spectral sequence has
$$E_{p,q}^2 = H_p(\Omega(S^2\times S^2); H_q\Omega S^3) \cong  H_p\Omega(S^2\times S^2) \tensor H_q\Omega S^3,$$
since the homology groups are torsion-free and the action of the fundamental group on the homology of the fiber is trivial.
In particular, $E_{*,*}^2$ is torsion-free with Poincar\'e series
$$\frac{1}{(1-t)^2(1-t^2)}.$$
Since this agrees with the Poincar\'e series of $H_*(\Omega M;\QQ)$, the spectral sequence collapses after tensoring with $\QQ$. But since $E_{*,*}^2$ is torsion-free, this implies that the spectral sequence collapses integrally. Moreover, since $E_{p,q}^\infty = E_{p,q}^2$ is free, the extensions relating $H_{p+q}\Omega M$ and $E_{p,q}^\infty$ split, yielding an isomorphism of graded abelian groups $H_*\Omega M \cong H_*\Omega(S^2\times S^2)\tensor H_*\Omega S^3$. In particular, $H_*\Omega M$ is torsion-free. Since the maps in the fibrations are maps of loop spaces, the spectral sequence is multiplicative, and the resulting filtration on $H_*(\Omega M)$ has associated graded ring $H_*\Omega(S^2\times S^2)\tensor H_*\Omega S^3 \cong \ZZ\langle \alpha,\beta,\gamma\rangle/(\alpha\beta + \beta\alpha, \alpha\gamma -\gamma\alpha,\beta\gamma-\gamma\beta)$.
\end{proof}

We are now in position to calculate the integral Pontryagin ring $H_*\Omega M$.

\begin{theorem}
\label{thm:pontryaginring}
The map $\ZZ\langle \alpha,\beta \rangle \cong H_*\Omega (S^2\vee S^2) \to H_*\Omega M$ is surjective. The kernel is generated by the classes $[[\alpha,\beta],\alpha]$ and $[[\alpha,\beta],\beta]$ as a two-sided ideal. Thus, there is an isomorphism of graded rings,
$$H_*\Omega M \cong \ZZ\langle \alpha, \beta \rangle / \big([[\alpha,\beta],\alpha], [[\alpha,\beta],\beta] \big).$$
\end{theorem}

\begin{proof}
As verified in Theorem \ref{thm:homotopy groups} above, the image of the generator $\pi_2\Omega S^3 \to \pi_2\Omega M$ is the Samelson product $[\alpha,\beta]$. It follows that the map $H_*\Omega S^3 \to H_*\Omega M$ sends the generator $\gamma$ to the commutator $[\alpha,\beta] = \alpha \beta +\beta \alpha$ with respect to the Pontryagin product.
As established in the previous lemma, there is a filtration of the ring $H_*\Omega M$ with associated graded isomorphic to $\ZZ\langle \alpha,\beta,\gamma\rangle/([\alpha,\beta],[\alpha,\gamma],[\beta,\gamma])$. It follows that $H_*\Omega M$ is generated by the classes $\alpha,\beta,\gamma$ under the Pontryagin product. Since $\gamma$ can be expressed as $[\alpha,\beta]$ it follows that already $\alpha$ and $\beta$ generate $H_*\Omega M$. Hence, $\ZZ\langle \alpha, \beta \rangle \cong H_*\Omega(S^2\vee S^2) \to H_*\Omega M$ is surjective.

Let $I\subset \ZZ\langle \alpha,\beta\rangle$ denote the two-sided ideal generated by $[[\alpha,\beta],\alpha]$ and $[[\alpha,\beta],\beta]$. It follows from Corollary \ref{cor:kernel} that $[[\alpha,\beta],\alpha]$ and $[[\alpha,\beta],\beta]$ map to zero in $H_*\Omega M$. There results a surjective ring homomorphism $\varphi\colon \ZZ\langle \alpha,\beta \rangle/I \to H_*\Omega M$.
A straightforward calculation shows that $\ZZ\langle \alpha,\beta\rangle / I$ is torsion-free with the same Poincar\'e series as $H_*\Omega M$. Every surjective map between finitely generated free abelian groups of the same rank is an isomorphism, so $\varphi$ must be an isomorphism.
\end{proof}

\begin{theorem}
The manifold $M$ is coformal over $\ZZ$.
\end{theorem}

\begin{proof}
As pointed out in \cite[Example 2.18]{BerglundBorjeson}, the manifold $M$ is coformal over $\QQ$ because its minimal model has quadratic differential.
To show that $M$ is coformal over $\ZZ$ we will apply Proposition \ref{prop:integral formality} to $A= C_*(\Omega M;\ZZ)$.

The Hochschild cohomology of
$$U = H_*\Omega M \cong \ZZ\langle \alpha, \beta \rangle / \big([[\alpha,\beta],\alpha], [[\alpha,\beta],\beta] \big).$$
is calculated in the next section using Theorem \ref{thm:koszulhochschild} (see Theorem \ref{moduledescriptiontheorem} and Remark \ref{remark:obstruction}). The calculation shows that the only non-vanishing obstruction group is $HH^2(U,U)(-3)$, and that this is isomorphic to $\big(U/[U,U]\big)_5$. This group is easily seen to be torsion-free. It follows that $M$ is coformal over $\ZZ$.
\end{proof}

\begin{corollary}
The cohomology $H^*(M;\ZZ)$ is a Koszul $A_\infty$-algebra, weakly equivalent to $C^*(M;\ZZ)$. The generators $a,b$ have weight $-1$, the classes $x,y$ have weight $-2$ and the top class $M$ has weight $-3$. The only non-zero higher operations are
$$m_3(a,b,b) = - m_3(b,b,a) = x,$$
$$m_3(a,a,b) = - m_3(b,a,a) = y.$$
The operations $m_n$ are zero for $n\geq 4$.
\end{corollary}

\begin{proof}
This follows from the proof of Theorem \ref{moduledescriptiontheorem} below together with Theorem \ref{thm:coformalkoszul} and Remark \ref{remark:PID}.
\end{proof}



\subsection{Hochschild cohomology computation}
In this section we compute the Hochschild cohomology of $U=H_*(\Omega M).$ As seen above this will enable us to conclude that $M$ is coformal over the integers and thus give us a description of the free loop homology of $M.$


We will describe the Hochschild cohomology $HH^*(U,U)$ as a module over its center $\mathcal{Z}(U)$. We begin by determining $\mathcal{Z}(U)$ explicitly.

\begin{proposition}
The center of the ring $U = \ZZ\langle \alpha,\beta \rangle/([[\alpha,\beta],\alpha], [[\alpha,\beta],\beta])$ is isomorphic to the polynomial ring,
$$\mathcal{Z}(U) \cong \ZZ[t_1,t_2,t_3],\quad |t_i| = 2,$$
generated by the three elements
$$t_1 = \alpha^2,\quad t_2= \beta^2,\quad t_3 = [\alpha,\beta].$$
\end{proposition}

\begin{proof}
To determine $\mathcal{Z}(U)$ we first note that $U$ have an additive basis given by elements $\beta^k(\alpha\beta)^\ell\alpha^m.$ Being in the center is equivalent to commuting with $\alpha$ and $\beta.$ Writing out the commutators we see that $\mathcal{Z}(U)$ has an additive basis given by elements $\alpha^{2p}((\alpha\beta)^q+(\beta\alpha)^q)\beta^{2r}.$ Thus $\mathcal{Z}(U)$ is generated freely as a commutative algebra by $\alpha^2,\beta^2$ and $[\alpha,\beta] =\alpha\beta+\beta\alpha.$
\end{proof}

Next, we determine the Hochschild cohomology as a module over $\mathcal{Z}(U)$.

\begin{theorem}
\label{moduledescriptiontheorem}

The Hochschild cohomology of $U$ is a module over $\mathcal{Z}(U).$

In weight $0$ the Hochschild cohomology is isomorphic as a $\mathcal{Z}(U)$-module to
$$\mathcal{Z}(U).$$

In weight $-1$ the Hochschild cohomology is isomorphic as a $\mathcal{Z}(U)$-module to
$$\mathcal{Z}(U)\{e_1,e_2,e_3,e_4,e_5,e_6\}/(2t_1e_1+t_3e_2,2t_2e_2+t_3e_1,t_3e_3-2t_1e_4+2t_2e_5-t_3e_6),$$ where $|e_1|=|e_2|=-2, |e_3|=|e_4|=|e_5|=|e_6|=-1.$ 

In weight $-2$ the Hochschild cohomology is isomorphic as a $\mathcal{Z}(U)$-module to
$$\mathcal{Z}(U)\{f_1,f_2,f_3,f_4,f_5,f_6\}/((4t_1t_2-t_3^2)f_1,(4t_1t_2-t_3^2)f_2,t_1t_2f_3+t_3f_4-t_2f_5-t_1f_6),$$ where $|f_1|=|f_2|=-5,|f_3|=-4,|f_4|=|f_5|=|f_6|=-2.$

In weight $-3$ the Hochschild cohomology is isomorphic to $s^{-7}U/[U,U]$ which as a $\mathcal{Z}(U)$-module is isomorphic to
$$\mathcal{Z}(U)\{g_1,g_2,g_3,g_4\}/(2t_1g_1,2t_2g_1,t_3g_1,t_3g_2-2t_1g_3,t_3g_3-2t_2g_2),$$ where $|g_1|=-7,|g_2|=-6,|g_3|=-6,|g_4|=-5.$

\end{theorem}

\begin{proof}

Let $C$ be the $A_\infty$-coalgebra defined as follows. As a $\mathbb{Z}$-module $$C \cong\mathbb{Z}\{1,a^*,b^*,x^*,y^*,M^*\}$$ where $|1|=0,|a^*|=|b^*|=2,|x^*|=|y^*|=5, |M^*|=7.$ The structure is determined by 

\begin{align*}
&\begin{aligned}
\Delta_2(a^*)&=1\otimes a^*+a^*\otimes 1, &\Delta_2(b^*)&=1\otimes b^*+b^*\otimes 1,\\
\Delta_2(x^*)&=1\otimes x^*+x^*\otimes 1, &\Delta_2(y^*)&=1\otimes y^*+y^*\otimes 1,\\
\end{aligned} \\
&\Delta_2(M^*)=a^*\otimes x^* + x^*\otimes a^* -b^*\otimes y^* - y^*\otimes b^*+1\otimes M^*+M^*\otimes 1,\\
&\begin{aligned}
\Delta_3(a^*)&=0, &\Delta_3(b^*)&=0, &\Delta_3(M^*)=0, \\
\end{aligned} \\
&\Delta_3(x^*)=a^*\otimes b^*\otimes b^*- b^*\otimes b^*\otimes a^*,\\
&\Delta_3(y^*)=a^*\otimes a^*\otimes b^*- b^*\otimes a^*\otimes a^*,
\end{align*} 

with the differential and all higher maps zero. There is a weight grading on $C$ such that $w(1)=0,$ $w(a^*)=w(b^*)=1$ $w(x^*)=w(y^*)=2$ and $w(M^*)=3.$ It is convenient to work with the linear dual $A_\infty$-algebra $A$ as well.
As a $\mathbb{Z}$-module,  $$A\cong\mathbb{Z}\{1,a,b,x,y,M\},$$ where the weights and degrees are inverted compared to $C.$ The structure maps are determined by 
\begin{align*}
m_2(a,x)&=m_2(x,a)=-m_2(b,y)=-m_2(y,b)=M,\\
m_2(a,y)&=m_2(y,a)=m_2(b,x)=m_2(x,b)=0,\\
m_3(a,a,b)&=-m_3(b,a,a)=y,\\
m_3(a,b,b)&=-m_3(b,b,a)=x,\\
m_3(a,a,a)&=m_3(a,b,a)=m_3(b,a,b)=m_3(b,b,b)=0,
\end{align*}
together with $m_3$ being zero if at least one argument is proportional to the unit and all other $m_i:$s being zero. 

If one writes down $\Omega C$ explicitly, one sees that it is described by $$
\big(\ZZ\langle \alpha,\beta,\xi,\zeta,\omega \rangle, \delta \big),
$$
with differential determined by $\delta \alpha = \delta \beta = 0$ and
$$
\delta \xi = [[\alpha,\beta],\beta], \quad \delta \zeta = [[\alpha,\beta],\alpha], \quad \delta \omega  = [\alpha,\xi] + [\beta,\zeta].
$$ 
One can now easily observe that $H_*(\Omega C)$ coincides with $U,$ the Pontryagin ring calculated in Theorem \ref{thm:pontryaginring}. We have the twisting morphism $\kappa_C:C\rightarrow U$ given by sending $a^*$ and $b^*$ to $\alpha$ and $\beta$ respectively. The twisted tensor product $C\otimes_{\kappa_C} U$ has differential described as follows.
\begin{equation*}
\begin{tabular}{ l c l  l c l}
$d_{\kappa}(1\otimes u)$& $=$& $0\hspace{2.5cm}$ &$d_{\kappa}(x^*\otimes u)$& $=$& $a^*\otimes\beta^2 u -b^*\otimes \beta \alpha u$ \\
$d_{\kappa}(a^*\otimes u)$& $=$& $1\otimes \alpha u$ &$d_{\kappa}(y^*\otimes u)$& $=$& $a^*\otimes \alpha\beta u-b^*\otimes \alpha^2 u$ \\
$d_{\kappa}(b^*\otimes u)$& $=$& $1\otimes \beta u$ &$d_{\kappa}(M^*\otimes u)$& $=$& $x^*\otimes \alpha u-y^*\otimes \beta u$ \\
\end{tabular}
\end{equation*} 
This is easily seen to be contractible, thus by Theorem \ref{thm:twistedtensorproductkoszul}, $C$ is a Koszul $A_\infty$-coalgebra. By Theorem \ref{thm:koszulhochschild}, the Hochschild cohomology of $U$ is given by the homology of $\Hom^{\kappa_C}(C,U).$ As a graded abelian group we have $\Hom^{\kappa_C}(C,U)\cong A\otimes U.$ Twisting $\Hom(C,A)$ with ${\kappa_C}$ is equivalent to twisting the tensor product of the $A_\infty$-algebras $A$ and $U$ with the element $\kappa=a\otimes\alpha+b\otimes\beta.$ The twisted differential $\mu_1^\kappa=\partial_\kappa$ act on generators by 
\begin{equation*}
\begin{tabular}{ l c l  l c l}
$\partial_\kappa(1\otimes u)$& $=$& $a\otimes[\alpha,u]+b\otimes[\beta,u]\hspace{1.5cm}$ &$\partial_\kappa(x\otimes u)$& $=$& $-M\otimes[\alpha,u]$ \\
$\partial_\kappa(a\otimes u)$& $=$& $-y\otimes[\beta,[\alpha,u]]$ &$\partial_\kappa(y\otimes u)$& $=$& $M\otimes[\beta,u]$ \\
$\partial_\kappa(b\otimes u)$& $=$& $x\otimes [\alpha,[\beta,u]]$ &$\partial_\kappa(M\otimes u)$& $=$& $0$ \\
\end{tabular} 
\end{equation*}
where the brackets are graded commutators. The differential respects the $\mathcal{Z}(U)$-module structure since elements of $\mathcal{Z}(U)$ pass through commutators, so the homology will be a module over $\mathcal{Z}(U).$ Here we already see that in weight $0,$ the Hochschild cohomology is isomorphic to $\mathcal{Z}(U)$ and in weight $3$ it is isomorphic to $s^{-7}U/[U,U].$

Now we want to describe $U$ as a $\mathcal{Z}(U)$-module. 
The complex $A\otimes U$ twisted with $\kappa$ is a free $\mathcal{Z}(U)$-module of rank $24.$ The following table describes the differential $\partial_\kappa=\mu^\kappa_1$. An element in the table is the differential of the element in the first column tensor the element in the top row.

\begin{tabular}{l | l  l  l  l }
 $\partial_\kappa=\mu^\kappa_1 \quad $             & 1 & $\alpha$ & $\beta$ & $\alpha\beta$   \\
\hline 
1 & 0 & $2\alpha^2a\otimes 1$ & $[\alpha,\beta]a\otimes 1$ &$2\alpha^2a\otimes \beta-[\alpha,\beta]a\otimes\alpha$    \\
   & &$+ [\alpha,\beta]b\otimes 1$ & $+2\beta^2b\otimes 1$ & $+[\alpha,\beta]b\otimes \beta-2\beta^2b\otimes\alpha$  \\
a  \rule{0pt}{2ex}  & 0 & 0 & 0 & $([\alpha,\beta]^2-4\alpha^2\beta^2)y\otimes 1$   \\
b   \rule{0pt}{2ex}& 0 & 0 & 0 & $([\alpha,\beta]^2-4\alpha^2\beta^2)x\otimes 1$  \\
x   \rule{0pt}{2ex}& 0 & $-2\alpha^2M\otimes 1$ &$ -[\alpha,\beta]M\otimes 1$ & $-2\alpha^2M\otimes \beta+[\alpha,\beta]M\otimes\alpha$  \\
y   \rule{0pt}{2ex}& 0 &  $[\alpha,\beta]M\otimes 1$ & $2\beta^2M\otimes 1$ & $[\alpha,\beta]M\otimes\beta-2\beta^2M\otimes\alpha$   \\
M  \rule{0pt}{2ex}& 0 & 0 & 0 & 0  \\
\end{tabular}
\\
From this table we see that the kernel is a weight graded $\mathcal{Z}(U)$-module with the following description.

In weight $0$ the kernel is generated by $1\otimes 1.$

In weight $-1$ the kernel is generated by the six elements $$a\otimes 1,b\otimes 1,a\otimes \alpha, a\otimes\beta,b\otimes \alpha,b\otimes\beta.$$

In weight $-2$ the kernel is generated by the $6$ elements $$x\otimes 1,y\otimes 1,x\otimes \beta+y\otimes \alpha,\beta^2x\otimes \alpha+\alpha^2y\otimes\beta,2\alpha^2y\otimes\alpha+[\alpha,\beta]x\otimes\alpha,2\beta^2x\otimes\beta+[\alpha,\beta]y\otimes\beta$$ with the relation $$\alpha^2\beta^2(x\otimes \beta+y\otimes \alpha)+[\alpha,\beta](\beta^2x\otimes \alpha+\alpha^2y\otimes\beta)-\alpha^2(2\beta^2x\otimes\beta+[\alpha,\beta]y\otimes\beta)$$ $$-\beta^2(2\alpha^2y\otimes\alpha+[\alpha,\beta]x\otimes\alpha)=0.$$

In weight $-3$ the kernel is generated by the $4$ elements
$$M\otimes 1,M\otimes \alpha,M\otimes \beta,M\otimes\alpha\beta.$$
Except in weight $-2,$ these are immediate. There we have to check which linear combinations can vanish in $M\otimes 1$ and we see that we have the elements $x\otimes1,y\otimes1,x\otimes \beta+y\otimes \alpha,\beta^2x\otimes \alpha+\alpha^2y\otimes\beta,2\alpha^2y\otimes\alpha+[\alpha,\beta]x\otimes\alpha,2\beta^2x\otimes\beta+[\alpha,\beta]y\otimes\beta$ span this part of the kernel. These are however not linearly independent but satisfy the identity $\alpha^2\beta^2(x\otimes \beta+y\otimes \alpha)+[\alpha,\beta](\beta^2x\otimes \alpha+\alpha^2y\otimes\beta)-\alpha^2(2\beta^2x\otimes\beta+[\alpha,\beta]y\otimes\beta)-\beta^2(2\alpha^2y\otimes\alpha+[\alpha,\beta]x\otimes\alpha)=0.$ 

Now determining the $\mathcal{Z}(U)$-module description of the homology is a matter of comparing the description of the kernel and the description of the image. We rename our generators as follows to get the presentation in the theorem.
$$e_1=a\otimes 1, e_2=b\otimes1,e_3=a\otimes \alpha, e_4=a\otimes \beta, e_5=b\otimes\alpha,e_6=b\otimes\beta$$
$$f_1=x\otimes 1,f_2=y\otimes 1,f_3=x\otimes \beta+y\otimes \alpha,f_4=\beta^2x\otimes \alpha+\alpha^2y\otimes\beta$$
$$f_5=2\alpha^2y\otimes\alpha+[\alpha,\beta]x\otimes\alpha,f_6=2\beta^2x\otimes\beta+[\alpha,\beta]y\otimes\beta$$
$$g_1=M\otimes 1,g_2=M\otimes \alpha, g_3=M\otimes\beta,g_4=M\otimes \alpha\beta$$

\end{proof}

\begin{remark}
\label{remark:obstruction}
Note that in weight $-3,$ the Hochschild cohomology is torsion free in even homological degrees. This is since the relations imposed by the differentials are all of the form $\alpha u= u\alpha$ and $\beta u= u \beta$ without any minus signs. This is useful in our discussion of coformality, in particular, the obstruction group $HH^2(U,U)(-3)$ is torsion free.
\end{remark}

Finally, we compute the algebra structure on the Hochschild cohomology.
\begin{theorem}
\label{thm:7manifoldhochschild}
The Hochschild cohomology $HH^*(U,U)$ is an algebra over $\mathcal{Z}(U)$. A presentation is given by the free graded commutative algebra over $Z(U)$ on generators $e_1,e_2,e_3,e_4,e_5,e_6,f_1,f_2,f_3$ and $f_4$ of degrees $|e_1|=|e_2|=-2, |e_3|=|e_4|=|e_5|=|e_6|=-1,|f_1|=|f_2|=-5,|f_3|=-4,|f_3|=-2,$ with relations imposed as follows. Firstly, all generators square to zero. Secondly, we have the relations

\begin{tabular}{ l  l  l  }
$2t_1e_1+t_3e_2=0$ & $t_3e_1+2t_2e_2=0$ & $t_3e_3+2t_2e_5=2t_1e_4+t_3e_6$ \\ $t_3^2f_1=4t_1t_2f_1$ & $t_3^2f_2=4t_1t_2f_2$ & $2t_1e_1f_1=0$ \\ 
$2t_2e_1f_1=0$ & $t_3e_1f_1=0$ & $t_3e_3f_1=2t_1e_4f_1$ \\ 
$2t_2e_3f_1=t_3e_4f_1$ &$t_1t_2f_3+t_3f_4=t_2e_3e_5+t_1e_4e_6$ & \\

\\

$e_1e_2=0$ & $e_1e_3+t_3f_2=0$ & $e_1e_4+2t_2f_2=0$ \\
$e_1e_5+t_3f_1=0$ & $e_1e_6+2t_2f_1=0$ & $e_2e_3=2t_1f_2$ \\
$e_2e_4=t_3f_2$ & $e_2e_5=2t_1f_1$ & $e_2e_6=t_3f_1$  \\
$e_3e_4=2t_2f_3-e_4e_6$ & $e_3e_5=2t_1f_3-e_5e_6$ & $e_3e_6=2f_4$ \\
$e_4e_5=t_3f_3$ & $e_1f_1+e_2f_2=0$ &  $e_1f_2=0$\\
$e_1f_3+e_4f_1=0$ &  $e_1f_4+t_2e_3f_1=0$& $e_2f_1=0$ \\
$e_2f_3+e_5f_2=0$&  $e_2f_4=t_1e_4f_1$& $e_3f_1+e_5f_2=0$ \\
$e_3f_2=0$ & $e_3f_3=e_6f_3$ & $e_3f_4+t_1t_2e_1f_1=0$ \\
$e_4f_1+e_6f_2=0$ & $e_4f_2=0$ & $e_4f_3+t_2e_1f_1=0$ \\
$e_4f_4+t_2e_3f_3=0$ & $e_5f_1=0$ & $e_5f_3=t_1e_1f_1$ \\
$e_5f_4+t_1e_3f_3=0$ & $e_6f_1=0$ & $e_6f_4=t_1t_2e_1f_1$ \\
$f_1f_2=0$ & $f_1f_3=0$ & $f_1f_4=0$ \\
$f_2f_3=0$ & $f_2f_4=0$ & $f_3f_4=0.$  \\
\\
\end{tabular}
Thirdly, the product of any three generators not all of the form $e_i$ is zero and the product of any four generators is also zero. 
\end{theorem}

\begin{proof}
From the definition of twisted $A_\infty$-algebra we obtain the following formulas. Together with the fact that $1\otimes 1$ acts as the identity and $M\otimes u$ is zero when multiplied with anything other than the identity they determine the multiplication.
\begin{equation*}
\begin{tabular}{ l c l }
$\mu^{\kappa}_2(a\otimes u_1,a\otimes u_2)$& $=$& $-y\otimes [\beta,u_1u_2]$  \\
$\mu^{\kappa}_2(b\otimes u_1,b\otimes u_2)$& $=$& $x\otimes[\alpha,u_1u_2]$ \\
$\mu^{\kappa}_2(a\otimes u_1,b\otimes u_2)$& $=$& $y\otimes [\alpha,u_1]u_2-(-1)^{|u_1|}x\otimes u_1[\beta,u_2]$  \\
$\mu^{\kappa}_2(x\otimes u_1,a\otimes u_2)$& $=$& $M\otimes u_1u_2$  \\
$\mu^{\kappa}_2(x\otimes u_1,b\otimes u_2)$& $=$& $0$  \\
$\mu^{\kappa}_2(y\otimes u_1,a\otimes u_2)$& $=$& $0$  \\
$\mu^{\kappa}_2(y\otimes u_1,b\otimes u_2)$& $=$& $-M\otimes u_1u_2$  \\
\end{tabular} 
\end{equation*}
Note that $\mu^{\kappa}_2$ is not associative on the nose but induces a commutative associative multiplication in homology. We see that $\mu^{\kappa}_2$ respects the $\mathcal{Z}(U)$-module structure so the homology will be an algebra over $\mathcal{Z}(U).$ 
We will use the notation given to the elements in the proof of Theorem \ref{moduledescriptiontheorem}.
Note that we have $e_1f_1=-e_2f_2=g_1,e_2f_3=e_3f_1=-e_5f_2=g_2, e_1f_3=-e_4f_1=e_6f_2=g_3,e_3f_3=e_6f_3=-g_4,e_3e_5=f_5$ and $e_4e_6=f_6.$ The other additive generators are primitive so we see that $e_1,e_2,e_3,e_4,e_5,e_6,f_1,f_2,f_3$ and $f_4$ generate the homology as a graded commutative algebra. The relations imposed come from two different sources. The first set comes from the $\mathcal{Z}(U)$-description in Theorem \ref{moduledescriptiontheorem}. The second set comes from writing out the products of all pairs of generators and comparing them.

There might be more relations coming from looking at products of three generators. It is easy to see that multiplying three generators gives zero unless all three are of the form $e_i.$ Since the generators square to zero we see that all the potentially non-zero such products are products of different generators. 

Now any such potentially nonzero triple reduced to a a scalar times a multiplication of two generators. Since all relations between such have been exhausted by the relations already written down we do not need to impose any more relations for these. 

Finally, it is easy to see that multiplying any four generators gives zero.
\end{proof}

\begin{corollary}
With the above description there is an algebra isomorphism $$H_{*+7}(LM)\cong HH^*(U,U).$$
\end{corollary}

\begin{proof}
This follows from Theorem \ref{stringtopologyformaltheorem}.
\end{proof}

\subsection*{Acknowledgments}
The first author was supported by the Swedish Research Council through grant no.~2015-03991.

\end{document}